\documentclass[a4paper,12pt]{amsart}
\usepackage[utf8]{inputenc}
\usepackage[centering]{geometry}
\usepackage{amsmath,amssymb,amsthm,amsfonts,mathrsfs}
\usepackage{mathtools}
\mathtoolsset{showonlyrefs}
\usepackage{graphicx}
\usepackage[T1]{fontenc}
\usepackage{booktabs}
\usepackage{xcolor}

\usepackage{romannum}

\usepackage{enumitem}

\usepackage{hyperref}
\hypersetup{
    colorlinks=true,
    linkcolor=blue,
    filecolor=magenta,      
    urlcolor=cyan,
    pdftitle={Overleaf Example},
    pdfpagemode=FullScreen,
    unicode=true
}

\theoremstyle{plain}
\newtheorem{theorem}{Theorem}[section]
\newtheorem{lemma}[theorem]{Lemma}
\newtheorem{proposition}[theorem]{Proposition}
\newtheorem{corollary}[theorem]{Corollary}

\theoremstyle{definition}
\newtheorem{definition}[theorem]{Definition}

\newtheorem{conjecture}[theorem]{Conjecture}

\theoremstyle{remark}
\newtheorem{remark}[theorem]{Remark}
\newtheorem{claim}[theorem]{Claim}

\title{Exceptional Loci and Hyperbolicity of Complements of Plane Curves with $\bar{c}_1^2 - \bar{c}_2 > 0$}
\author{Wei Chen}

\address{ Department of Mathematics and Physics \\
Roma Tre University   \\
Largo San Leonardo Murialdo, I-00146 \\
Rome, Italy}
\email{wei.chen@uniroma3.it, weichen97.ag@gmail.com}

\begin{document}

\pagenumbering{arabic}

\begin{abstract}
Let $B$ be a plane curve which has simple normal crossings. We prove that, if $B=B_1 \cup B_2$ has two irreducible components with $\deg B_1, \deg B_2 \geq 5$ or $\deg B_1 = 4, \deg B_2 \geq 7$, and if $B$ is general, then the associated algebraic exceptional set is empty.

Based on this, we show that for a general plane curve $B$ satisfying $\bar{c}_1^2 - \bar{c}_2 > 0$, the open surface $\mathbb{P}^2 \setminus B$ is Kobayashi hyperbolic.  
\end{abstract}
\subjclass[2020]{14H20, 14H45, 14J26}
\keywords{Algebraic exceptional set, hyperbolicity, rational curves.}

\maketitle

\tableofcontents 

\section{Introduction}
\label{sec:intro}
From the arithmetic point of view, Lang and Vojta conjectured that the rational (integral) points on a smooth (quasi-)projective variety of (log-)general type defined over (the ring of integers of) a number field are not Zariski dense, see \cite{Lang86} and \cite{vojta1987diophantine} for more details. In particular, working over $\mathbb{C}$, this motivates the following geometric conjecture:
\begin{conjecture}\label{conj: non-log-general-type subvariety}
    Let $(X,D)$ be a log-smooth variety, and $NG(X,D)$ be the union of subvarieties $Z\subset X$ such that $(Z, Z\cap D)$ is not of log-general type. Then $\overline{NG(X,D)}$ is a proper closed subvariety of $X$, where $\overline{NG(X,D)}$ is the Zariski closure of $NG(X,D)$.
\end{conjecture}

For precise definitions of all the notions involved in Conjecture~\ref{conj: non-log-general-type subvariety}, see Notations and Conventions in the beginning of Section~\ref{sec:prelim}.

This conjecture is widely open, even for surfaces. On surfaces, the only subvarieties that need to be considered are curves, and $NG(X,D)$ can be read in a more explicit way, which gives the definition of the so-called (Lang's) algebraic exceptional set:
\begin{definition}\label{def: algebraic exceptional set}
    Let $(S,B)$ be a log-smooth surface, then its algebraic exceptional set is
    $$
    \mathcal{E}(S,B) \coloneqq \left\{ C\subset S\ \vert\ C \text{ is an integral curve } \colon 2g(C)-2+\#\nu_C^{-1}(C\cap B) \leq 0 \right\},
    $$
    where $\nu_C\colon \Tilde{C} \to C$ is the normalization map. When the surface $S$ is clear from the context, we denote this set simply as $\mathcal{E}(B)$.
\end{definition}

\begin{remark}\label{rmk: exceptional set for ample boundary}
    If $B$ is ample in Definition~\ref{def: algebraic exceptional set}, then 
    $$
    \mathcal{E}(S,B) = \left\{ C\subset S\ \vert\ C \text{ is an integral curve, }\Tilde{C}\cong\mathbb{P}^1, \text{ and } \#\nu_C^{-1}(C\cap B) \leq 2 \right\}.
    $$
\end{remark}

Clearly, for log-smooth surfaces, Conjecture~\ref{conj: non-log-general-type subvariety} reads as:
\begin{conjecture}\label{conj: finite exceptional set}
    For a log-smooth surface $(S,B)$ of log-general type, $\mathcal{E}(B)$ is finite.
\end{conjecture}

When $B=\emptyset$ and $S$ is a smooth projective surface of general type with $c_1^2>c_2$, Bogomolov proved this conjecture in \cite{Bogomolov77}. More precisely, Bogomolov proved that curves of fixed genus on such a surface form a bounded family. In particular, such a surface has only finitely many rational and elliptic curves. 

When $B\neq \emptyset$, Conjecture~\ref{conj: finite exceptional set} is still widely open, even for $\mathbb{P}^2$. For plane curves, as a consequence of the logarithmic Kobayashi
conjecture\cite{Kobayashi70ConjectureGenericEmptiness}, we further have the following:
\begin{conjecture}\label{conj: generic empty exceptional set}
    Let $(\mathbb{P}^2,B)$ be log-smooth. If $\deg B \geq 5$, then $\mathcal{E}(B)=\emptyset$ for a general $B$.
\end{conjecture}
If $\deg B=4$ then $\mathcal{E}(B) \neq \emptyset$, as it contains the bitangent lines to $B$. If we assume $B$ to be very general, then Conjecture~\ref{conj: generic empty exceptional set} can be deduced from the work of Xu \cite{Xu96IntersectionVeryGeneralCurve}, where he proved that for a very general plane curve $B$ of degree $b\geq 3$, any other integral curve intersects $B$ in at least $b-2$ distinct points. It can also be deduced from the work of Xi Chen in \cite{XiChen04} and the work of Pacienza--Rousseau in \cite{PacienzaRousseau+2007+221+235} on logarithmic algebraic hyperbolicity. Furthermore, for a very general quartic curve, in \cite{chen2023algebraichyperbolicitycomplementsgeneric} the authors proved that the algebraic exceptional set consists exactly of bitangent lines and flex lines. However, Conjecture~\ref{conj: generic empty exceptional set} itself is still open. 

Historically, it turns out that the more irreducible components $B$ has, the easier the Conjectures~\ref{conj: finite exceptional set} and ~\ref{conj: generic empty exceptional set} are. If $B$ has at least $4$ irreducible components, then they become easy problems; see, for instance, \cite[Proposition~4.3.1]{caporaso2024hypertangencyplanecurvesalgebraic}. If $B$ has three irreducible components, then Conjecture~\ref{conj: finite exceptional set} is solved but requires a non-trivial proof, see \cite{CZ08}, \cite{guo2023vojtasabcconjecturealgebraic}. In the three-component case, recently, Caporaso and Turchet provided a new proof in \cite{caporaso2024hypertangencyplanecurvesalgebraic} which is purely algebro-geometric, and their proof drops the assumption of rationality of the curves in $\mathcal{E}(\mathbb{P}^2,B)$. More precisely, they considered the following set of the so-called hyperbitangent curves
$$
    \operatorname{Hyp}(B;2) \coloneqq \left\{ C\subset S\ \vert\ C \text{ is an integral curve, } \#\nu_C^{-1}(C\cap B) \leq 2 \right\},
$$
proved that $\operatorname{Hyp}(B;2) = \mathcal{E}(B)$, provided an effective bound of the cardinality of this set, and proved Conjecture~\ref{conj: generic empty exceptional set} in this case. Their method has been generalized, first to Hirzebruch surfaces in \cite{chen2025exceptionalsetHirzebruch}, and then to arbitrary surfaces in \cite{caporaso2026exceptionallocialgebraicsurfaces}. 

If $B$ has at most two irreducible components, only very few things were known. When $B$ has two irreducible components, Brotbek and Deng \cite{BrotdekDeng18} showed that when both irreducible components of $B$ have the same degrees larger than $2^{12}$, the augmented base locus of the logarithmic cotangent bundle is $\mathbb{B}_+(\Omega(\log B)) = B$ for a general $B$, and Conjecture~\ref{conj: generic empty exceptional set} follows from the fact that every curve $C \in \mathcal{E}(B)$ is contained in $\mathbb{B}_+(\Omega(\log B))$. When $B$ is irreducible (and smooth), Conjecture~\ref{conj: generic empty exceptional set} can be deduced from the works of Siu--Yeung \cite{SiuYeung96onecomponent} and Siu \cite{Siu15generichyperbolicity} for $\deg B \gg 0$; and some examples were constructed via arithmetic methods, see \cite{chen2026langvojtadegeneration} and the references there in.

In this paper, we study Conjecture~\ref{conj: generic empty exceptional set} for a two-component plane curve $B=B_1\cup B_2$. Our main result is the following theorem:
\begin{theorem}[Theorem~\ref{thm:TwoComponentGenericEmpty}]\label{thm: main 1}
    Let $B=B_1\cup B_2$ be a reduced plane curve with two irreducible components which has simple normal crossings, $b_1 \coloneqq \deg B_1 \leq b_2 \coloneqq \deg B_2$, and $b\coloneqq b_1 + b_2$. 
    
    If $b_2\geq b_1 \geq 5$ or $b_1=4$, $b_2\geq 7$, then for a general $B$, the algebraic exceptional set $\mathcal{E}(B)$ is empty.
\end{theorem}

The assumptions on the degrees here are given by the log-Chern class inequality $\Bar{c}_1^2-\Bar{c}_2>0$, see Lemma~\ref{lem:degree of two component curve with c1^2-c2>0}. Our proof follows the strategy of Pacienza--Rousseau's work \cite{PacienzaRousseau+2007+221+235} for very general boundary divisors in $\mathbb{P}^n$. To remove the very general assumption, we will need, as just mentioned, the log-Chern class inequality $\Bar{c}_1^2-\Bar{c}_2>0$. This brings us back to the fundamental work of Bogomolov \cite{Bogomolov77}. After Bogomolov's breakthrough, it has been conjectured that there is a uniform version of Bogomolov's result for any (minimal) smooth projective surface of general type with $c_1^2-c_2>0$, see, for instance, \cite[Section~2, the very last remark]{Tian96Cetraro}. This was finally achieved by Miyaoka in \cite{Miyaoka08EffectiveBogomolov}. Later, Sabatino generalized Miyaoka's work to the logarithmic setting in \cite{Sabatino22LogEffectiveBogomolov}, which provides us with the tool we need.

Recall that a complex manifold $Y$ is said to be \textit{Brody hyperbolic} if there is no non-constant holomorphic map $f\colon \mathbb{C} \to Y$, and it is said to be \textit{Kobayashi hyperbolic} if the Kobayashi pseudo-distance is a distance. It is well-known that Kobayashi hyperbolicity implies Brody hyperbolicity, and Brody proved that for a compact complex manifold, the two notions are equivalent \cite{Brody78brodyhyperbolicity}. For more details on hyperbolicity, we refer the readers to \cite{Kobayashi98hyperbolicspace}.

Combining Theorem~\ref{thm: main 1} with the work of Caporaso--Turchet \cite{caporaso2026exceptionallocialgebraicsurfaces} and a result of El Goul \cite{ElGoul03AlgebricDegeneracy}, we have the following:
\begin{theorem}[Theorem~\ref{thm: Brody hyperbolicity of complement of plane curve with c1^2-c2>0}]
    Let $(\mathbb{P}^2,B)$ be a log-smooth surface with $\bar{c}_1^2 - \bar{c}_2>0$. Then for a general $B$, the surface $\mathbb{P}^2\setminus B$ is Brody hyperbolic.
\end{theorem}

With a little bit more work and applying a famous theorem of Green \cite{Green77hyperbolicity}, we prove:
\begin{theorem}[Theorem~\ref{thm: Kobayashi hyperbolicity of complement of plane curve with c1^2-c2>0}]
    Let $(\mathbb{P}^2,B)$ be a log-smooth surface with $\bar{c}_1^2 - \bar{c}_2>0$. Then for a general $B$, the surface $\mathbb{P}^2\setminus B$ is hyperbolically embedded in $\mathbb{P}^2$. In particular, it is Kobayashi hyperbolic.
\end{theorem}

It is worth mentioning that in a recent preprint, for a general two-component plane curve of certain lower degrees, using 2-jet differentials, Xie and Zhao \cite{XieZhao2026kobayashihyperbolicitygeneralsurfaces} proved that the complement open surface is hyperbolically embedded in $\mathbb{P}^2$ and hence is Kobayashi hyperbolic.

The paper is organized as follows:
In Section~\ref{sec:prelim}, we collect the necessary preliminary results. In Section~\ref{sec:algebraic exceptional set and hyperbolicity}, we prove our main results.

\subsection*{Acknowledgements}
The author would like to thank his supervisor Lucia Caporaso for constant support during the development of this paper. The author thanks Gianluca Pacienza, Erwan Rousseau for valuable conversations. The author
thanks Damian Brotbek, Ya Deng, and Gian Pietro Pirola for helpful discussions. The author thanks Amos Turchet for helpful comments. The author also thanks the organizers of the workshop “Topology of Algebraic Varieties” at Technische Universität Chemnitz during March 10-13, 2026 for providing a stimulating working environment. The author is partially supported by PRIN 2022L34E7W, and is a member of the INdAM group GNSAGA.

\subsection*{Declaration of Use of AI}
The proof of Proposition~\ref{prop: Global Generation of Twisted Tangent Bundle} and the first two paragraphs of the proof of Theorem~\ref{thm: Non Zero Section of Twisted Log Canonical} were written interactively with the assistance of Gemini 3.1 Pro. The author takes full responsibility for the correctness of the mathematical content.

\section{Preliminaries}
\label{sec:prelim}
\subsection*{Notations and Conventions} We work over $\mathbb{C}$. 
\begin{itemize}
    \item Let $B=\sum B_i$ be an effective divisor on a smooth projective variety $X$ of dimension $n$. We say $B$ has simple normal crossings (or $B$ is a simple normal crossing divisor) if $B$ is reduced, each component $B_i$ is smooth, and $B$ is defined in a neighborhood of any point by an equation in local analytic coordinates of the type $z_1 \cdots z_k = 0$ for some $k\leq n$.
    \item We say $(S,B)$ is a log-smooth surface if $S$ is a smooth projective surface, and $B = \sum_{i=1}^{r} B_i$ is a reduced effective divisor on $S$ which has simple normal crossings with $r$ distinct irreducible components $B_i$, $i=1,\dots,r$.
    \item On a smooth projective variety $S$, the Picard group $\operatorname{Pic}(S)$ is canonically identified with the divisor class group modulo linear equivalence $\operatorname{Div}(S)/\sim$. By a slight abuse of notation, for a divisor $D \in \operatorname{Div}(S)$ and a line bundle $\mathcal{L}$ over $S$, we will write $\mathcal{L}+D$ for the corresponding class in $\operatorname{Pic}(S) \cong \operatorname{Div}(S)/\sim$.
    \item Let $B = \sum_{i=1}^{r} B_i$ be a reduced effective divisor on a smooth projective variety $X$. We say $B$ is general if it satisfies a certain property $\mathcal{P}$ such that there is a non-empty Zariski open subset $U$ in $\lvert B_1\rvert \times \cdots \times \lvert B_r\rvert$ with all the members in $U$ satisfying $\mathcal{P}$ and $B\in U$.
    \item For any $m\in \mathbb{N}$, we denote by $\mathbb{P}^{N_m} \coloneqq \mathbb{P}\left( H^0(\mathbb{P}^2 , \mathcal{O}_{\mathbb{P}^2}(m)) \right)$ the linear system parameterizing degree $m$ plane curves. It has dimension $(m^2+3m)/2$.
\end{itemize}

In this section, we collect the known results that we will need.

By direct computation, we have the following lemma:
\begin{lemma}[\cite{Rousseau21degreeboundoftwocomponentcurve}, Corollary~2.2]\label{lem:degree of two component curve with c1^2-c2>0}
    Let $S = \mathbb{P}^2$ and $B = \cup_{i=1}^{r} B_i$ a simple normal crossing curve where $B_i$ is a smooth curve of degree $b_i$, $b_1 \leq b_2 \leq \cdots \leq b_r$. Then its log-Chern classes satisfy $\bar{c}_1^2 - \bar{c}_2 > 0$ if and only if
    $r\geq 5$; or, $r\geq 4$ and $b_4\geq 2$; or, $r \geq 3$ and $b_1 \geq 2 , b_3 \geq 3$ or $b_1=1 , b_2\geq 3 , b_3\geq 4$; or, $r=2$ and $b_1 \geq 5$ or $b_1=4 , b_2\geq 7$.
\end{lemma}

By generalizing Miyaoka's work \cite{Miyaoka08EffectiveBogomolov} on smooth projective surface of general type with $c_1^2-c_2 > 0$ to the logarithmic setting, Sabatino proved the following:

\begin{theorem}[\cite{Sabatino22LogEffectiveBogomolov}, Corollary~1.3]\label{thm:LogCanonicalBound-Sabatino}
    Let $(S,B)$ be a log-smooth surface and $C$ be an irreducible curve on $S$ which is not contained in $B$. If $K_S+B$ is big and nef and moreover $\bar{c}_1^2-\bar{c}_2>0$, then 
    $$
    (K_S+B) \cdot C \leq \mathbf{A} \left( 2g(C) - 2 + \#\nu_C^{-1}(C\cap B) \right) + \mathbf{B},
    $$
    where $\mathbf{A}$ and $\mathbf{B}$ are two constants depending only on $\bar{c}_1^2$ and $\bar{c}_2$.
\end{theorem}

\begin{corollary}[\cite{Sabatino22LogEffectiveBogomolov}, Corollary~1.4]\label{cor: UniformFinitenessExceptionalSet-Sabatino}
    Let $(S,B)$ be a log-smooth surface. If $K_S+B$ is big and nef and moreover $\bar{c}_1^2-\bar{c}_2>0$, then the irreducible curves $C$ on $S$ that are not contained in $B$ and such that $2g(C)-2+\#\nu_C^{-1}(C\cap B)$ is bounded form a bounded family, where the number of components is bounded by a function of $\bar{c}_1^2-\bar{c}_2$. In particular, the algebraic exceptional set $\mathcal{E}(S,B)$ is finite and its cardinality is bounded by a function of $\bar{c}_1^2$ and $\bar{c}_2$.
\end{corollary}

We will also need the following two results:
\begin{proposition}[\cite{ElGoul03AlgebricDegeneracy}, Corollary~2.4.3]\label{prop:AlgebraicDegeneracyOfEntireCurve-ElGoul}
    Let $(S,B)$ be a log-smooth surface of log-general type with $\bar{c}_1^2-\bar{c}_2>0$. Then every entire curve $f \colon \mathbb{C} \to S\setminus B$ is algebraically degenerate, that is, $f(\mathbb{C})$ lies in an algebraic curve in $S$.
\end{proposition}

\begin{theorem}[\cite{Xu98finitenessofrationalcurves}, Theorem~1]\label{thm : xu98-finiteness-of-rational-curves-in-a-linear-system}
    Let $B$ be a smooth curve on a smooth projective surface $S$ with geometric genus $g(B) > 0$. If $\mathcal{L} \in \operatorname{Pic}(S)$ is a line bundle with $(\mathcal{L} \cdot B) > 0$, then there are only a finite number of integral curves $D \in \lvert \mathcal{L} \rvert$ with geometric genus
    $$
    g(D) < \frac{1}{2}\left( (K_S + B)\cdot \mathcal{L}\right) + 1 
    $$
    such that $D$ intersects the curve $B$ at exactly one point (set-theoretically).
\end{theorem}

\section{Algebraic Exceptional Set and Hyperbolicity}
\label{sec:algebraic exceptional set and hyperbolicity}
Let $Z = [Z_0:Z_1:Z_2]$ be the homogeneous coordinates on $\mathbb{P}^2$. Let $X = [X_\alpha]$ and $Y = [Y_\beta]$ be the coordinates on $\mathbb{P}^{N_{b_1}}$ and $\mathbb{P}^{N_{b_2}}$ respectively, where $\alpha = (\alpha_0, \alpha_1, \alpha_{2}) \in \mathbb{N}^3$, $\beta=(\beta_0, \beta_1, \beta_2) \in \mathbb{N}^3$, $\lvert \alpha \vert = \sum \alpha_i=b_1$ and $\lvert \beta \rvert = \sum \beta_i =b_2$.

Let $\mathscr{X} \subset \mathbb{P}^2 \times \mathbb{P}^{N_{b_1}} \times \mathbb{P}^{N_{b_2}}$ be the universal hypersurface of degree $(b_1+b_2, 1, 1)$ given by the equation
$$
\left(\sum_{\lvert \alpha \rvert = b_1} X_{\alpha} Z^{\alpha} \right) \left(\sum_{\lvert \beta \rvert = b_2} Y_{\beta} Z^{\beta}\right) 
= 0
$$
where $[Z] \in \mathbb{P}^2$, $[X_{\alpha}] \in \mathbb{P}^{N_{b_1}}$, $[Y_{\beta}] \in \mathbb{P}^{N_{b_2}}$, $Z^{\alpha} = Z_0^{\alpha_0} Z_1^{\alpha_1} Z_2^{\alpha_2}$ and similarly for $Z^{\beta}$.

Following the approach of Pacienza--Rousseau\cite{PacienzaRousseau+2007+221+235}, we obtain the following Proposition~\ref{prop: Global Generation of Twisted Tangent Bundle} and Theorem~\ref{thm: Non Zero Section of Twisted Log Canonical}:

\begin{proposition}\label{prop: Global Generation of Twisted Tangent Bundle}
The twisted logarithmic tangent bundle 
$$T_{\mathbb{P}^2 \times \mathbb{P}^{N_{b_1}} \times \mathbb{P}^{N_{b_2}}}(-\log \mathscr{X})(1,0,0)$$ 
is generated by its global sections.
\end{proposition}

\begin{proof}
To trivialize the twist and maintain a smooth ambient space, we move to $\mathbb{P}^4$ with homogeneous coordinates $[Z_0:Z_1:Z_2:Z_3:Z_4]$. Consider the projective space $\mathbb{P}^{N_{b_1}+1}$ with coordinates $\left[ [X_{\alpha}] \colon X_{new}\right]$ and $\mathbb{P}^{N_{b_2}+1}$ with coordinates $\left[ [Y_{\alpha}] \colon Y_{new}\right]$. Let $(X_i\neq 0)$ be the canonical open affine space in $\mathbb{P}^{N_{b_1}+1}$ and similarly for $(Y_i\neq 0)$. We define an open set $U = U_X \times U_Y \subset \mathbb{P}^{N_{b_1}+1} \times \mathbb{P}^{N_{b_2}+1}$, where 
$$
U_X \coloneqq (X_{new} \neq 0) \cap \left( \bigcup_{\lvert \alpha \rvert = b_1} (X_{\alpha} \neq 0) \right) \subset \mathbb{P}^{N_{b_1}+1}, 
$$
and
$$
U_Y \coloneqq (Y_{new} \neq 0) \cap \left( \bigcup_{\lvert \beta \rvert = b_2} (Y_{\beta} \neq 0) \right) \subset \mathbb{P}^{N_{b_2}+1}. 
$$

Let $\mathscr{Z} \subset \mathbb{P}^4 \times U$ be the complete intersection defined by the following two independent equations
$$F_1 = X_{new}Z_3^{b_1} + \sum_{\lvert \alpha \rvert=b_1} X_\alpha Z^\alpha = 0,
\quad
F_2 = Y_{new}Z_4^{b_2} + \sum_{\lvert \beta \rvert=b_2} Y_\beta Z^\beta = 0.$$
Because $F_1$ and $F_2$ depend on disjoint sets of parameters and independent auxiliary variables ($Z_3$ and $Z_4$), they intersect transversally, hence $\mathscr{Z}$ is smooth.

Consider the natural projection $\pi: \mathscr{Z} \to \mathbb{P}^2 \times \mathbb{P}^{N_{b_1}} \times \mathbb{P}^{N_{b_2}}$ dropping the auxiliary variables $Z_3, Z_4, X_{new}, Y_{new}$. Then $\mathscr{H} \coloneqq \mathscr{Z} \cap \{Z_3 Z_4 = 0\} = \pi^{-1}(\mathscr{X})$.

Therefore, we obtain a dominant log-morphism 
$$
\pi: (\mathscr{Z}, \mathscr{H}) \to (\mathbb{P}^2 \times \mathbb{P}^{N_{b_1}} \times \mathbb{P}^{N_{b_2}}, \mathscr{X}),
$$ 
which induces a surjective map
$$
\pi_* \colon
\overline{T}_{\mathscr{Z}}(1,0,0)\coloneqq T_{\mathscr{Z}}(-\log\mathscr{H})(1,0,0)
\longrightarrow
T_{\mathbb{P}^2 \times \mathbb{P}^{N_{b_1}} \times \mathbb{P}^{N_{b_2}}}(-\log \mathscr{X})(1,0,0).
$$

To prove the statement of the proposition, it suffices to prove that $\overline{T}_{\mathscr{Z}}(1,0,0)$ is globally generated.

Consider the open set $U_0 = \{Z_0 \neq 0\} \times U$ with the induced inhomogeneous coordinates $z_1, z_2, z_3, z_4$. Normalizing $X_{new} = Y_{new} = 1$, the equations defining $\mathscr{Z}$ on $U_0$ become $\mathscr{Z}_0 = \{ f_1 = 0 \} \cap \{ f_2 = 0 \}$, where
$$f_1 = z_3^{b_1} + \sum_{\lvert \alpha \rvert \le b_1} X_\alpha z^\alpha \quad \text{and} \quad f_2 = z_4^{b_2} + \sum_{\lvert \beta \rvert \le b_2} Y_\beta z^\beta .$$
Note that after the restriction, here $\alpha = (\alpha_1,\alpha_2) \in \mathbb{N}^2$ and $\beta = (\beta_1,\beta_2) \in \mathbb{N}^2$.
The divisor becomes $\mathscr{H}_0 = \mathscr{Z}_0 \cap \{z_3 z_4 = 0\}$.

For the parameter space directions, consider the vector fields
$$V_{\alpha,j}^{(X)} = \frac{\partial}{\partial X_\alpha} - z_j \frac{\partial}{\partial X_{\tilde{\alpha}}}$$
where $\alpha = (\alpha_1, \alpha_{2}) \in \mathbb{N}^2$, $j \in \{1, 2\}$ such that $\alpha_j \geq 1$, $\tilde{\alpha}_k = \alpha_k$ if $k \neq j$ and $\tilde{\alpha}_j = \alpha_j - 1$. This trivially satisfies $V_{\alpha,j}^{(X)}(f_1) = 0$ and $V_{\alpha,j}^{(X)}(f_2) = 0$. 

We symmetrically construct fields $V_{\beta,k}^{(Y)}$ for the $Y$-parameters. Notice that these are logarithmic vector fields of $(\mathscr{Z}_0, \mathscr{H}_0)$ which extend to $(\mathscr{Z}, \mathscr{H})$ with a pole order equal to 1. 

Consider a vector field on the base space
$$V_Z = \sum_{j=1}^2 v_j \frac{\partial}{\partial z_j}$$
where $v_j = \sum_{k=1}^2 v_k^{(j)} z_k + v_0^{(j)}$ is linear in the variables $z_k$. We extend this to the auxiliary variables:
$$V_0 = V_Z + v_3 z_3 \frac{\partial}{\partial z_3} + v_4 z_4 \frac{\partial}{\partial z_4}$$
where $v_3, v_4 \in \mathbb{C}$. The presence of $z_3$ and $z_4$ in the coefficients ensures $V_0$ is logarithmic with respect to $\mathscr{H}_0$. Furthermore, because its coefficients are at most linear in the affine coordinates, $V_0$ extends to a global vector field on $\mathbb{P}^4$.

We claim that there exist vertical corrections $V_X = \sum_{\lvert\alpha\rvert \leq b_1} V_\alpha^{(X)} \frac{\partial}{\partial X_\alpha}$ and $V_Y = \sum_{\lvert\beta\rvert \leq b_2} V_\beta^{(Y)} \frac{\partial}{\partial Y_\beta}$ such that the combined vector field $V = V_X + V_Y + V_0$ is tangent to $\mathscr{Z}_0$. 

Indeed, since $f_1$ does not depend on the variables $Y_\beta$ or $z_4$, applying $V$ to $f_1$ yields
$$ V(f_1) = \sum_{\lvert \alpha \rvert \le b_1} V_\alpha^{(X)} z^\alpha + \sum_{\alpha, j} X_\alpha v_j \frac{\partial z^\alpha}{\partial z_j} + b_1 v_3 z_3^{b_1}. $$
To evaluate this on the subvariety $\mathscr{Z}_0$, we substitute $z_3^{b_1} = - \sum X_\alpha z^\alpha$, which gives
$$ V(f_1)\Big|_{\mathscr{Z}_0} = \sum_{\lvert \alpha \rvert \le b_1} V_\alpha^{(X)} z^\alpha + \sum_{\alpha, j} X_\alpha v_j \frac{\partial z^\alpha}{\partial z_j} - b_1 v_3 \sum_{\lvert \alpha \rvert \le b_1} X_\alpha z^\alpha = 0. $$

Because the base coefficients $v_j = \sum_{k=1}^2 v_k^{(j)} z_k + v_0^{(j)}$ are linear in the spatial variables $z_k$, the derivative operation $v_j \frac{\partial}{\partial z_j}$ maps any monomial $z^\alpha$ to a linear combination of other monomials of degree at most $b_1$. Thus, we can expand and group the sum by the resulting monomials $z^\mu$ to obtain
$$ \sum_{\alpha, j} X_\alpha v_j \frac{\partial z^\alpha}{\partial z_j} = \sum_{\lvert \mu \rvert \le b_1} L_\mu(X) z^\mu, $$
where each coefficient $L_\mu(X)$ is a homogeneous linear polynomial in the parameters $X$. Replacing the index $\mu$ with $\alpha$, we obtain
$$ \sum_{\lvert \alpha \rvert \le b_1} \Big[ V_\alpha^{(X)} - b_1 v_3 X_\alpha + L_\alpha(X) \Big] z^\alpha = 0. $$
This uniquely determines our vertical corrections on the affine chart as
$$ V_\alpha^{(X)} = b_1 v_3 X_\alpha - L_\alpha(X). $$
Therefore, the coefficients $V_\alpha^{(X)}$ are homogeneous linear polynomials in $X$. This guarantees that the vertical correction $V_X = \sum_{\lvert\alpha\rvert \le b_1} V_\alpha^{(X)} \frac{\partial}{\partial X_\alpha}$ extends to a globally defined vector field on $\mathbb{P}^{N_{b_1}+1}$. Consequently, it restricts to a valid, smooth vector field on the open subset $U_X$.

Similarly, we can solve the tangency condition $V(f_2) = 0$ on $\mathscr{Z}_0$. By an identical computation substituting $z_4^{b_2} = - \sum Y_\beta z^\beta$, we obtain
$$ \sum_{\lvert \beta \rvert \le b_2} V_\beta^{(Y)} z^\beta + \sum_{\beta, j} Y_\beta v_j \frac{\partial z^\beta}{\partial z_j} - b_2 v_4 \sum_{\lvert \beta \rvert \le b_2} Y_\beta z^\beta = 0, $$
and we uniquely choose $V_\beta^{(Y)}$ as suitable linear combinations of the parameters $Y$ to cancel the coefficients of $z^\beta$, and $V_Y$ gives a smooth vector field on $U_Y$.

Because $f_1$ and $f_2$ depend on disjoint sets of parameters, a single choice of $V_0$ on the base uniquely determines both required vertical corrections without interference. By construction, this logarithmic vector field $V$ of $(\mathscr{Z}_0, \mathscr{H}_0)$ extends to $(\mathscr{Z}, \mathscr{H})$. 

The vector fields $V_{\alpha,j}^{(X)}$, $V_{\beta,k}^{(Y)}$, and $V$ give the global generation of $\overline{T}_{\mathscr{Z}}(1,0,0)$.
\end{proof}

\begin{theorem}\label{thm: Non Zero Section of Twisted Log Canonical}
    Let $B=B_1 \cup B_2$ be a reduced plane curve with two irreducible components such that $B$ has simple normal crossings, $b_1 \coloneqq \deg B_1 \leq b_2 \coloneqq \deg B_2$, and $b\coloneqq b_1 + b_2$. 
    
    Fix $d\in \mathbb{N}$. For a general $B$, any integral curve $D\subset \mathbb{P}^2$ of degree $d$ which is not contained in $B$ satisfies
    $$
    h^0\left( \tilde{D} , (K_{\tilde{D}}+\nu_{D}^{-1}(D\cap B))\otimes \nu_{D}^* \mathcal{O}_{\mathbb{P}^2}(4-b) \right) 
    \neq 0.
    $$
    Here by a slight abuse of notation, we also use $\nu_D$ to denote the composition of the normalization map $\tilde{D} \to D$ with the embedding $D \hookrightarrow \mathbb{P}^2$.
\end{theorem}
\begin{proof}
    Let $U\subset \mathbb{P}^{N_{b_1}} \times \mathbb{P}^{N_{b_2}}$ be the Zariski open dense subset parameterizing $(B_1,B_2)$ such that $B_1+B_2$ has simple normal crossings. Now consider
        $
        \mathbb{P}^{N_d} \times U
        $
        which parameterizes the pair $(D,B=B_1+B_2)$. Let $\mathcal{H}_d \subseteq \mathbb{P}^{N_d} \times U$ be the locus where the statement of the theorem fails. To prove the theorem, it suffices to show that the projection map $\pi_U \colon \mathcal{H}_d \to U$ is not dominant.

        Assume the contrary, then by passing to suitable subvarieties, we may assume without loss of generality that $\dim \mathcal{H}_d = \dim U$. By passing to the \'etale locus of $\pi_U \colon \mathcal{H}_d \to U$, we may assume that we are considering an irreducible subvariety
        $
        \mathscr{D} \subset \mathbb{P}^2 \times U
        $
        such that the projection map $\pi_d\colon \mathscr{D} \to U$ is of relative dimension $1$ and its fibers are degree $d$ curves. Let $\mathscr{P}$ be the family of divisors induced by the intersections $D \cap B$ for $(D,B) \in \mathcal{H}_d$, and let $(\tilde{\mathscr{D}},\tilde{\mathscr{P}}) \to (\mathscr{D},\mathscr{P})$ be a log resolution of $(\mathscr{D} , \mathscr{P})$.

        Denote $N\coloneqq N_{b_1}+N_{b_2}$.
        
        Take a general $B \in U$, let $D$ be its corresponding fiber of $\pi_d$ which is a degree $d$ curve, and $(\tilde{D}, \tilde{P})$ be the corresponding fiber of the log resolution. By construction, we have $\tilde{P}=\nu_D^{-1}(D\cap B)$. Denote the fiber of the projection map $\mathbb{P}^2\times U \to U$ over $B$ by $\mathbb{P}_B^2 \coloneqq \mathbb{P}^2\times \{B\}$.
        
        Then
        \begin{equation}\label{equation1}
            K_{\tilde{D}}+\tilde{P} \cong \Omega_{\tilde{\mathscr{D}}}(\log \tilde{\mathscr{P}})^{N+1} \Big\vert_{\tilde{D}}.
        \end{equation}
        Indeed,
        $$
        \Omega_{\tilde{\mathscr{D}}}(\log \tilde{\mathscr{P}})^{N+1} \Big\vert_{\tilde{D}}
        = \Omega_{\tilde{\mathscr{D}}}^{N+1} \otimes \mathcal{O}_{\tilde{\mathscr{D}}}(\tilde{\mathscr{P}}) \Big\vert_{\tilde{D}}
        = \Omega_{\tilde{\mathscr{D}}}^{N+1}\Big\vert_{\tilde{D}} \otimes \mathcal{O}_{\tilde{D}}(\tilde{P}),
        $$
        and by the adjunction formula
        $$
        K_{\tilde{D}} 
        = (K_{\tilde{\mathscr{D}}} \otimes \det \mathcal{N}_{\tilde{D}/\tilde{\mathscr{D}}}) \Big\vert_{\tilde{D}}
        = \Omega_{\tilde{\mathscr{D}}}^{N+1}\Big\vert_{\tilde{D}}.
        $$
        since $K_{\tilde{\mathscr{D}}} = \Omega_{\tilde{\mathscr{D}}}^{N+1}$ and $\mathcal{N}_{\tilde{D}/\tilde{\mathscr{D}}}$ is trivial as the normal bundle of a fiber in a family.
        
        And by construction we have
        \begin{equation}\label{equation2}
            T_{\mathbb{P}^2 \times \mathbb{P}^{N_{b_1}} \times \mathbb{P}^{N_{b_2}}}(- \log \mathscr{X}) \Big\vert_{\mathbb{P}_{B}^2} \otimes (K_{\mathbb{P}_B^2}+B)
            \cong
            \Omega_{\mathbb{P}^2 \times \mathbb{P}^{N_{b_1}} \times \mathbb{P}^{N_{b_2}}}(\log \mathscr{X})^{N+1} \Big\vert_{\mathbb{P}_B^2}.
        \end{equation}

        We have a generically surjective map $\Omega_{\mathbb{P}^2 \times \mathbb{P}^{N_{b_1}} \times \mathbb{P}^{N_{b_2}}}(\log \mathscr{X}) \to \Omega_{\tilde{\mathscr{D}}}(\log \tilde{\mathscr{P}})$, which induces a map
        $$
        \Omega_{\mathbb{P}^2 \times \mathbb{P}^{N_{b_1}} \times \mathbb{P}^{N_{b_2}}}(\log \mathscr{X})^{N+1} \Big\vert_{\mathbb{P}_B^2}(4-b)
        \longrightarrow
        \Omega_{\tilde{\mathscr{D}}}(\log \tilde{\mathscr{P}})^{N+1}\Big\vert_{\tilde{D}}(4-b)
        =
        (K_{\tilde{D}}+\tilde{P})\otimes\nu_D^*\mathcal{O}_{\mathbb{P}^2}(4-b).
        $$
        Since $B$ is chosen to be general in $U$, we may assume that this map is non-zero.

        Recall that $K_{\mathbb{P}^2}+B = \mathcal{O}_{\mathbb{P}^2}(b-3)$. By \eqref{equation2}, we obtain
        \begin{equation*}
            \begin{split}
                \Omega_{\mathbb{P}^2 \times \mathbb{P}^{N_{b_1}} \times \mathbb{P}^{N_{b_2}}}(\log \mathscr{X})^{N+1} \Big\vert_{\mathbb{P}_B^2}(4-b)
                & \cong
                T_{\mathbb{P}^2 \times \mathbb{P}^{N_{b_1}} \times \mathbb{P}^{N_{b_2}}}(- \log \mathscr{X}) \Big\vert_{\mathbb{P}_{B}^2} \otimes \mathcal{O}_{\mathbb{P}^2}(1)
                \\
                & =
                T_{\mathbb{P}^2 \times \mathbb{P}^{N_{b_1}} \times \mathbb{P}^{N_{b_2}}}(-\log \mathscr{X})(1,0,0) \Big\vert_{\mathbb{P}_B^2}.
            \end{split}
        \end{equation*}
        
        By Proposition~\ref{prop: Global Generation of Twisted Tangent Bundle}, we know $\Omega_{\mathbb{P}^2 \times \mathbb{P}^{N_{b_1}} \times \mathbb{P}^{N_{b_2}}}(\log \mathscr{X})^{N+1} \Big\vert_{\mathbb{P}_B^2}(4-b)$ is globally generated. Then the non-zero map
        $$
        \Omega_{\mathbb{P}^2 \times \mathbb{P}^{N_{b_1}} \times \mathbb{P}^{N_{b_2}}}(\log \mathscr{X})^{N+1} \Big\vert_{\mathbb{P}_B^2}(4-b)
        \longrightarrow
        (K_{\tilde{D}}+\tilde{P})\otimes\nu_D^*\mathcal{O}_{\mathbb{P}^2}(4-b)
        $$
        produces a non-zero section of $(K_{\tilde{D}}+\tilde{P})\otimes\nu_D^*\mathcal{O}_{\mathbb{P}^2}(4-b)$, contradiction.
\end{proof}

\begin{theorem}\label{thm:TwoComponentGenericEmpty}
    Let $B=B_1\cup B_2$ be a reduced plane curve with two irreducible components such that $B$ has simple normal crossings, $b_1 \coloneqq \deg B_1 \leq b_2 \coloneqq \deg B_2$, and $b\coloneqq b_1 + b_2$. 
    
    If $b_2\geq b_1 \geq 5$ or $b_1=4$, $b_2\geq 7$, then for a general $B$, the algebraic exceptional set $\mathcal{E}(B)$ is empty.
\end{theorem}
\begin{proof}
    By Lemma~\ref{lem:degree of two component curve with c1^2-c2>0}, we notice that the assumption $b_2\geq b_1 \geq 5$ or $b_1=4$, $b_2\geq 7$ is equivalent to $\bar{c}_1^2 - \bar{c}_2 > 0$. Then $K_{\mathbb{P}^2} + B = \mathcal{O}_{\mathbb{P}^2}(b-3)$ is (very) ample. Therefore, by Theorem~\ref{thm:LogCanonicalBound-Sabatino}, we know there exists a positive integer $N(b_1,b_2) \in \mathbb{N}$ depending only on $b_1$ and $b_2$ such that for any curve $C \in \mathcal{E}(B)$, we have $\deg C \leq N(b_1,b_2)$.

    Now, we choose $B$ to be general such that Theorem~\ref{thm: Non Zero Section of Twisted Log Canonical} holds for all curves of degree $\leq N(b_1,b_2)$. Then, for any integral curve $D$ of degree $d \leq N(b_1,b_2)$, the logarithmic canonical divisor $K_{\tilde{D}}+\nu_D^{-1}(D \cap B)$ is ample. Indeed, we have
    $$
    K_{\tilde{D}}+\nu_D^{-1}(D \cap B) = \left((K_{\tilde{D}}+\nu_D^{-1}(D \cap B)) \otimes \nu_{D}^*\mathcal{O}_{\mathbb{P}^2}(4-b)  \right) \otimes \nu_D^* \mathcal{O}_{\mathbb{P}^2}(b-4).
    $$
    Since $b>4$, we know $\nu_D^* \mathcal{O}_{\mathbb{P}^2}(b-4)$ is ample. By Theorem~\ref{thm: Non Zero Section of Twisted Log Canonical}, we know $(K_{\tilde{D}}+\nu_D^{-1}(D \cap B)) \otimes \nu_{D}^*\mathcal{O}_{\mathbb{P}^2}(4-b)$ has a non-zero global section, hence it has non-negative degree. Therefore, $K_{\tilde{D}}+\nu_D^{-1}(D \cap B)$ is ample. By construction, curves in $\mathcal{E}(B)$ are not of log-general type, hence this implies that $\mathcal{E}(B) = \emptyset$.
\end{proof}

In \cite{caporaso2024hypertangencyplanecurvesalgebraic}, Caporaso--Turchet proved that for a plane curve $B$ of $\deg B\geq 5$ with at least three irreducible components, if $B$ is general, then $\mathcal{E}(B)$ is empty. By Lemma~\ref{lem:degree of two component curve with c1^2-c2>0} and El Goul's result Proposition~\ref{prop:AlgebraicDegeneracyOfEntireCurve-ElGoul}, from \cite{caporaso2024hypertangencyplanecurvesalgebraic} and Theorem~\ref{thm:TwoComponentGenericEmpty}, we obtain the following:
\begin{theorem}\label{thm: Brody hyperbolicity of complement of plane curve with c1^2-c2>0}
    Let $(\mathbb{P}^2,B)$ be a log-smooth surface with $\bar{c}_1^2 - \bar{c}_2>0$. Then for a general $B$, the surface $\mathbb{P}^2\setminus B$ is Brody hyperbolic.
\end{theorem}

To prove the Kobayashi hyperbolicity of $\mathbb{P}^2\setminus B$ in Theorem~\ref{thm: Brody hyperbolicity of complement of plane curve with c1^2-c2>0}, we need the following lemma, which is inspired by the proof of \cite[Theorem~F]{XieZhao2026kobayashihyperbolicitygeneralsurfaces}:

\begin{lemma}\label{lem : hyperbolically embedding criterion}
    Let $(\mathbb{P}^2,B)$ be a log-smooth surface of log-general type. If $\mathbb{P}^2 \setminus B$ is Brody hyperbolic, then it is hyperbolically embedded in $\mathbb{P}^2$. In particular, it is Kobayashi hyperbolic.
\end{lemma}

\begin{proof}
    The log-general type assumption is equivalent to $\deg B \geq 4$. Denote $B = \cup_{i=1}^r B_i$, where the $B_i$'s are the irreducible (smooth) components of $B$; denote $B_i^{\circ} \coloneqq B_i \setminus \left( \cup_{j\neq i} B_j \right)$; denote $b\coloneqq \deg B$ and $b_i \coloneqq \deg B_i$.

    By Green's theorem \cite[Theorem~2]{Green77hyperbolicity}, it suffices to show that $\forall i \in \{1,\cdots,n\}$, $B_i^{\circ}$ is of log-general type. Indeed, we have
    \begin{equation}
        \begin{split}
            \deg K_{B_i} + \# B_i\setminus B_i^{\circ}
            & = \deg K_{B_i} + \sum_{j\neq i} (B_i\cdot B_j)
            \\
            & = 2 p_g(B_i) - 2 + (B_i \cdot B) - B_i^2
            \\
            & = b_i(b_i - 3) + b_ib - b_i^2
            \\
            & = b_i(b-3) >  0
        \end{split}
    \end{equation}
    as desired.
\end{proof}

As an immediate consequence of Theorem~\ref{thm: Brody hyperbolicity of complement of plane curve with c1^2-c2>0} and Lemma~\ref{lem : hyperbolically embedding criterion}, we obtain:

\begin{theorem}\label{thm: Kobayashi hyperbolicity of complement of plane curve with c1^2-c2>0}
    Let $(\mathbb{P}^2,B)$ be a log-smooth surface with $\bar{c}_1^2 - \bar{c}_2>0$. Then for a general $B$, the surface $\mathbb{P}^2\setminus B$ is hyperbolically embedded in $\mathbb{P}^2$. In particular, it is Kobayashi hyperbolic.
\end{theorem}

\begin{remark}
    When $B\subset \mathbb{P}^2$ has at least three irreducible components and all the components have small degrees, the log-Chern class inequality $\bar{c}_1^2-\bar{c}_2>0$ is not satisfied. Nevertheless, in these cases, for a general $B$ with $\deg B \geq 5$, the surface $\mathbb{P}^2\setminus B$ is still Kobayashi hyperbolic. Indeed, by \cite{caporaso2024hypertangencyplanecurvesalgebraic} and \cite[Theorem~1]{ru2025campanasorbifoldconjecturenumerically} we know $\mathbb{P}^2\setminus B$ is Brody hyperbolic, then by Lemma~\ref{lem : hyperbolically embedding criterion} we know it is Kobayashi hyperbolic. See \cite{caporaso2026exceptionallocialgebraicsurfaces} for further discussions.
\end{remark}

At last, we provide an alternative proof of the finiteness of $\mathcal{E}(B)$ for any simple normal crossing plane curve $B=B_1 \cup B_2$ with $b_2\coloneqq \deg B_2\geq b_1\coloneqq \deg B_1 \geq 5$ or $b_1=4$, $b_2\geq 7$:

\begin{theorem}
    Let $(\mathbb{P}^2,B)$ be log-smooth, and $B=B_1 \cup B_2$ with $b_2\coloneqq \deg B_2\geq b_1\coloneqq \deg B_1 \geq 5$ or $b_1=4$, $b_2\geq 7$. Then $\mathcal{E}(B)$ is a finite set.
\end{theorem}
\begin{proof}
    Let 
    $$
    \mathcal{E}(B,1) \coloneqq 
        \left\{
            C \subset \mathbb{P}^2 
            \ \vert \
            C \text{ is a rational curve such that }
            \#\nu_{C}^{-1}(C \cap B) = 1
        \right\}
    $$
    and
    $$
    \mathcal{E}(B,2) \coloneqq 
        \left\{
            C \subset \mathbb{P}^2 
            \ \vert \
            C \text{ is a rational curve such that }
            \#\nu_{C}^{-1}(C \cap B) = 2
        \right\},
    $$
    then we clearly have $\mathcal{E}(B) = \mathcal{E}(B,1) \cup \mathcal{E}(B,2)$.

    Firstly, we note that $\mathcal{E}(B,1) = \emptyset$. Indeed, suppose $D$ is a curve in $\mathcal{E}(B,1)$ of degree $d$, then $D \cap B_i \neq \emptyset$ for any $i=1,2$. Hence $D\cap B = \{p\}$ for some point $p\in B_1\cap B_2$, and $p$ is a unibranch point of $D$. Then, $D$ must be transverse to one of $B_1$ and $B_2$ at $p$. If $D$ intersects $B_1$ transversally at $p$, then we have 
    $$
    db_1 = (D\cdot B_1)
    = (D \cdot B_1)_{p}
    = \operatorname{mult}_{p}(D)
    \leq d,
    $$
    which implies that $b_1 \leq 1$, contradiction. By symmetry, the same argument applies to the case if $D$ intersects $B_2$ transversally at $p$. Hence $\mathcal{E}(B,1)=\emptyset$.

    Now let's consider $\mathcal{E}(B,2)$. Let $D$ be a (rational) curve in $\mathcal{E}(B,2)$, there are four possible types of $D$:
    \begin{enumerate}[label=(\roman*)]
        \item\label{curve : type1} $\# (D \cap B) = 1$;
        \item\label{curve : type2} $\# (D \cap B) = 2$, and $D\cap (B_1\cap B_2) = \emptyset$;
        \item\label{curve : type3} $\# (D \cap B) = 2$, and $\# (D\cap (B_1\cap B_2)) = 1$;
        \item\label{curve : type4} $\# (D \cap B) = 2$, and $D\cap B \subset B_1\cap B_2$.
    \end{enumerate}

    \begin{claim}
        For any $d\in \mathbb{N}$, there exist only finitely many curves of each type in $\mathcal{E}(B,2) \cap \lvert \mathcal{O}(d) \rvert$.
    \end{claim}

    By Theorem~\ref{thm:LogCanonicalBound-Sabatino}, we know there exists a positive integer $N(b_1,b_2) \in \mathbb{N}$ depending only on $b_1$ and $b_2$ such that for any curve $D \in \mathcal{E}(B)$, we have $\deg D \leq N(b_1,b_2)$. Then the theorem follows from the claim.

    It remains to prove the claim.

    For curves of type~\ref{curve : type1}:
    A curve $D$ of type~\ref{curve : type1} intersects $B_1$ at only one point set-theoretically. Since $b_1 \geq 4$, for the right hand side of the inequality in Theorem~\ref{thm : xu98-finiteness-of-rational-curves-in-a-linear-system}, we have 
    $$
    \frac{1}{2} \left( (K_{\mathbb{P}^2} + B_1) \cdot \mathcal{O}(d) \right) + 1 = \frac{1}{2}d(b_1-3)+1 > 0.
    $$ 
    Now the statement immediately follows from Theorem~\ref{thm : xu98-finiteness-of-rational-curves-in-a-linear-system}.

    For curves of type~\ref{curve : type2}:
    By assumption, we see that for a type~\ref{curve : type2} curve $D$, we must have $\#(D\cap B_1) = \#(D\cap B_2) = 1$. Now the same argument as in the type~\ref{curve : type1} case applies.
     
    For curves of type~\ref{curve : type3}:
    Assume that $D \cap B = \{p,q\}$, $p \in B_1\cap B_2$ and $q\not\in B_1\cap B_2$. If $q \in B_2\setminus B_1$, then $D\cap B_1 = \{p\}$, and the same argument as in the proof of type~\ref{curve : type1} curves case applies. Similarly, if $q \in B_1 \setminus B_2$, then $D\cap B_2 = \{p\}$, and again the same argument as in the proof of type~\ref{curve : type1} curves case applies since $b_2 \geq b_1 \geq 4$.

    For curves of type~\ref{curve : type4}:
    Let $D$ be a curve of type~\ref{curve : type4} of degree $d$. Assume that $D\cap B = \{p,q\}\subset B_1\cap B_2$. Then both $p$ and $q$ are unibranch points of $D$. Let $m_p \coloneqq \operatorname{mult}_{p}(D)$ and $m_q \coloneqq \operatorname{mult}_{q}(D)$.

    Then, $D$ cannot be transverse to $B_i$ at both $p$ and $q$, for any $i=1,2$. Indeed, assume $D$ meets $B_i$ at $p$ and $q$ transversally for some $i\in\{1,2\}$, then we have
    $$
    db_i = (D \cdot B_1) 
    = (D\cdot B_1)_p + (D\cdot B_1)_q
    = m_p + m_q
    \leq 2d,
    $$
    which implies that $b_i \leq 2$, contradiction.

    Without loss of generality, we may assume that $D$ is tangent to $B_1$ at $p$ and is tangent to $B_2$ at $q$. Then, $D$ is transverse to $B_1$ at $q$. Now, consider the blow-up surface $S = \operatorname{Bl}_{q}(\mathbb{P}^2)$. We denote the blow-up morphism as $\pi \colon S \to \mathbb{P}^2$, the exceptional divisor as $E$, and the divisor $\pi^*\mathcal{O}(1)$ as $H$. In $S$, let $\Tilde{D} \in \lvert dH - m_q E\rvert$ be the strict transform of $D$ and $\Tilde{B}_1 \in \lvert b_1 H - E \rvert$ be the strict transform of $B_1$, then $\Tilde{D}$ intersects $\Tilde{B}_1$ at the point $p$ only. For the right hand side of the inequality in Theorem~\ref{thm : xu98-finiteness-of-rational-curves-in-a-linear-system}, we have
    \begin{multline*}
        \frac{1}{2}\left( (K_{S} + \Tilde{B}_1) \cdot \Tilde{D} \right) + 1
        = \frac{1}{2}\left[ \left(-3H + E + b_1 H - E\right) \cdot (dH - m_q E)\right] + 1 \\
        = \frac{1}{2} \left[ (b_1 - 3)H \cdot (dH - m_q E) \right] + 1
        = \frac{1}{2}d(b_1 - 3) + 1 > 0.
    \end{multline*}
    By Theorem~\ref{thm : xu98-finiteness-of-rational-curves-in-a-linear-system}, we see there are only finitely many rational curves in $\lvert d H - m_q E \rvert$ that intersects $\Tilde{B}_1$ at one point set-theoretically.

    Since there are only finitely many possible values of $m_q$---more precisely, $1 \leq m_q \leq d$---it follows that there are only finitely many curves of type~\ref{curve : type4} of degree $d$ that are tangent to $B_1$ at $p$ and to $B_2$ at $q$. Since there are only finitely many choices for $p$ and $q$, we conclude that there are only finitely many curves of type~\ref{curve : type4} of degree $d$.

    Hence the claim is established and the theorem is proved.
\end{proof}

\bibliographystyle{alpha} 
\bibliography{ref}

\end{document}